\documentclass[11pt,a4paper,oneside]{article}
\usepackage[margin=1in]{geometry}
\usepackage[usenames,dvipsnames]{xcolor} 
\usepackage[pdftitle={},
pdfauthor={Zachary Burton},
colorlinks=true,linkcolor=Cyan,urlcolor=Blue,citecolor=NavyBlue,bookmarks=true,bookmarksopen=true,bookmarksopenlevel=2,unicode=true,linktocpage]{hyperref}

\usepackage[utf8]{inputenc}
\usepackage[T1]{fontenc}
\usepackage[english]{babel}
\usepackage{authblk}  
\usepackage{amsmath}  
\usepackage{graphicx}  
\usepackage{amsthm} 
\usepackage[capitalize]{cleveref} 
\usepackage{amsmath,bm}  
\usepackage{amsfonts} 
\usepackage{mathtools} 
\usepackage{dsfont}
\usepackage{autonum} 

\newtheorem{thm}{Theorem}[section]

\newtheorem{prop}[thm]{Proposition}
\newtheorem{lem}[thm]{Lemma}

\theoremstyle{definition}
\newtheorem{defn}[thm]{Definition}

\theoremstyle{remark}
\newtheorem{rem}[thm]{Remark}

\usepackage{todonotes}

\def\P{\mathbb{P}}

\title{Poissonization-based collision threshold derivation for random walks on lattices}

\author{Zachary Burton\footnote{\href{mailto:zacnwo@mit.edu}{zacnwo@mit.edu}}}
\affil{MIT}

\makeatletter
\newcommand{\subjclass}[2][1991]{%
	\let\@oldtitle\@title%
	\gdef\@title{\@oldtitle\footnotetext{#1 \emph{Mathematics subject classification.} #2.}}%
}
\newcommand{\keywords}[1]{%
	\let\@@oldtitle\@title%
	\gdef\@title{\@@oldtitle\footnotetext{\emph{Key words and phrases.} #1.}}%
}
\makeatother



\begin{document}

\maketitle

\begin{abstract}
In this expository note, we give a short derivation of the expected number of collisions between two independent simple random walkers on integer lattices. Adapting a Poissonization technique introduced by Lange, we express the collision probability as the return probability of the continuous-time difference walk, given by a modified Bessel function. Analyzing its asymptotic decay yields a clean, self-contained proof that the expected number of collisions in $\mathbb{Z}^d$ is finite if and only if $d\geq3$. We also provide a general formula for the asymptotic number of collisions.
\end{abstract}

\tableofcontents

\section{Introduction}

Random walks are a fundamental object of study in probability theory, and their collision properties reveal deep connections between geometry, dimension, and stochastic behavior. One classical question asks how often two random walks meet. In dimensions 1 and 2, collisions between random walkers on integer lattices occur infinitely often, which is a standard result of Pólya \cite{polya1921}. However, in higher dimensions, there are a finite number of collisions, due to the transience of these integer lattices. We  show in Theorem \ref{thm: theorem} that at dimension 3 and upwards, collisions between walks are finite, due to the convergence of the expectation as an integral.

Our goal in this paper is to re-derive this threshold behavior by using a self-contained, analytic approach, whilst remaining accessible to undergraduate readers with formal exposure to probability theory and analysis. By modeling the difference random walk in continuous time, which we emphasize in \textbf{Section 2} and applying a Poissonization technique, we obtain an explicit expression for the collision probability in terms of a modified Bessel function, which are defined in \textbf{Section 3}, with a result which allows analysis of our final integral. Analyzing its long-term decay in \textbf{Section 4} yields specific integrability criteria that allow us to distinguish between the finite and infinite collision cases. 

This method was found in Lange’s \cite{lange2015} Poissonization-based derivation of return probabilities for a single random walk, resulting in the Bessel function that Novak \cite{novak2013polyasrandomwalktheorem} presented in his 2013 note. We adapt Lange's approach to the two-walk collision setting, expanding on its link between discrete-time and continuous-time random walks, and providing previously omitted explanations, such as the link between continuous and discrete walks, and justification for some intermediate steps, including the Laplace method to ignore higher order terms in the asymptotic analysis. We also compute the specific leading constant for the asymptotic collision expectation decay in $\mathbb{Z}^d$ in Theorem \ref{thm: constant}.

The implications of collisions of random walks are highly non-trivial; we can use these walks to analyze Brownian motion \cite{MR2180635}, \cite{MR4364738}, and the voter model \cite{astoquillca2024stationarymeasuresvariantsvoter}, which has applications in opinion dynamics and epidemiology.

Now we state the main results of this note.

\begin{thm}\label{thm: theorem}
    The expected number of collisions of two simple random walkers in $\mathbb{Z}^d$ is infinite for d $\leq$ 2 and finite for $d>2$.
\end{thm}

\begin{thm}\label{thm: constant}
The expected number of collisions of two simple random walkers in $\mathbb{Z}^d$ at time $t$ satisfies
\[
\mathbb{E}\biggl[\#\bigl\{t: X(t)= Y(t) \bigr\}\biggr] \sim \biggl(\frac{d}{\pi}\biggr)^{d/2} \cdot \frac{1}{t^{d/2}}
\quad \text{as} \quad t \to \infty.
\]
\end{thm}

\section{Setup and the difference walk}

We consider the standard lattice graph in \( \mathbb{Z}^d \), where the set of vertices are \( V = \mathbb{Z}^d \), and edges connect pairs of vertices at Euclidean distance one. That is,
\[
E = \left\{ (x, y) \in \mathbb{Z}^d \times \mathbb{Z}^d : \|x - y\| = 1 \right\},
\]
where \( \| \cdot \| \) denotes the usual Euclidean norm:
\[
\|x - y\| = \left( \sum_{i=1}^d (x_i - y_i)^2 \right)^{1/2}.
\]

\begin{defn} A \textbf{simple discrete-time random walk} on \( \mathbb{Z}^d \) is a Markov process \( (X_n)_{n \geq 0} \) such that, at each step, the walker moves uniformly to one of the \( 2d \) nearest neighbors. That is,
\[
\mathbb{P}(X_{n+1} = y \mid X_n = x) =
\begin{cases}
\frac{1}{2d} & \text{if } \|x - y\| = 1, \\
0 & \text{otherwise}.
\end{cases}
\]
\end{defn}
We consider two such walks, \( (X_n) \) and \( (Y_n) \), which are independent. Specifically, that their joint distribution factorizes:
\[
\P(X_0 = x_0, \ldots, X_n = x_n,\, Y_0 = y_0, \ldots, Y_n = y_n)
= \P(X_0 = x_0, \ldots, X_n = x_n) \cdot \P(Y_0 = y_0, \ldots, Y_n = y_n).
\]
In particular, given their respective positions, the next steps of \(X_n\) and \(Y_n\) are chosen independently.

We are interested in the expected number of collisions, that is, the expected number of times \(n\) when \(X_n = Y_n\). Defining the difference walk \(D_n := X_n - Y_n\), we have \(D_n = 0\) precisely when the walkers collide. Thus, the expected number of collisions in discrete-time is:

\begin{equation}\label{eq:collision-expect-discrete}
\mathbb{E}_n\biggl[ \# \bigl\{n: X_n = Y_n \bigr\}\biggr] 
\end{equation}

As the probability of the event that $D_n = 0$ is a Bernoulli random variable, to compute the expectation, we sum each individual discrete case of n. Due to the continuous-time case forming the basis of our analysis, this sum will appear as an integral in the main theorem.
The next section sets up continuous-time walks and expresses the collision expectation in terms of a closed-form solution.

To facilitate the analysis, we now introduce the continuous-time analog of the walks, which were first proposed by Montroll and Weiss \cite{montroll1965}. 

\begin{defn}
    A \textbf{continuous-time random walk} is a process where the walker waits a time distributed Exponential(1), before moving uniformly to any one of its $2d$ neighbors. 
\end{defn}  

Each walk jumps according to a Poisson random process with rate 1, which differs from the discrete-time walk which necessarily jumps at integer times. We now show that the expected number of collisions in discrete-time equals the expected number of collisions in continuous time. 

\begin{lem}
    Let $X_n$ and $Y_n$ be discrete-time random walks on $\mathbb{Z}^d$, and 
    denote the expected number of collisions in discrete-time as $\mathcal{D}$ and the expected number of collisions of their continuous-time analogs as $\mathcal{C}$.

    We have $\mathcal{D} = \mathbb{E}\bigl[\#\bigl\{n: X_n = Y_n \bigr\}\bigr]$ and $\mathcal{C} = \mathbb{E}\bigl[\#\bigl\{t: X(t)= Y(t) \bigr\}\bigr]$.

    Then, $\mathcal{D} = \mathcal{C}$ almost surely.
\end{lem}

\begin{proof}
Denote the exact time of the jump of any random walk as $T_k$. If $X_k=Y_k$ for some k, we know neither walk jumps until the next jump process $T_{k+1}$. Then as $I_k = \bigl[T_k, T_k+1\bigr)$ where $I_k$ is the time the two walks occupy the same position, we see $I_k \subseteq \bigl\{t: X(t)= Y(t) \bigr\}$. Thus, every discrete collision contributes one connected component to the continuous collision set.

Conversely, let $\mathcal{J}$ be a connected component of the continuous collision set. As the walks jump only at Poisson times, $\mathcal{J}$ must be contained in some interval $I_k = \bigl[T_k, T_k+1\bigr)$. At the left end-point $T_k$, the walks must occupy the same position, so $X_k = Y_k$. Thus $\mathcal{J}$ came from a discrete collision at step k. Thus, every connected component of the continuous collision set comes from a collision.
\end{proof}

Thus, we can define the difference walk $D(t) = X(t) - Y(t)$ where $X(t)$ and $Y(t)$ are independent continuous-time simple random walks on $\mathbb{Z}^d$. The process $D(t)$ then also takes values in $\mathbb{Z}^d$ and evolves independently component-wise.

\begin{prop}
    In each coordinate direction $j \in \{1, 2, \dots, d\}$, the walk $D_{j}(t)$ is a continuous-time symmetric random walk whose jump times form a Poisson process with rate $\frac{2}{d}$.
\end{prop}
\begin{proof}

Since X(t) and Y(t) are independent continuous-time random walks each jumping with rate 1, their difference $D(t) = X(t) - Y(t)$ jumps with rate 2. 

At each jump, exactly one coordinate j is chosen uniformly in $\{1 \dots d\}$ so a jump in coordinate j occurs with probability $\frac{1}{d}$. By Poisson thinning introduced in lectures \cite[Prop 1.3]{lastandpenrose2017}, each coordinate process is an \textbf{independent} Poisson process of rate $\frac{2}{d}$.
\end{proof}

As each coordinate process has jump rate $\frac{2}{d}$, the total number of jumps $N_j(t)$ in period $[0,t]$ is \text{Poisson}($\frac{2t}{d}$). Now we express $\mathbb{P}\bigl(D_{j}(t)=0\bigr)$ in terms of a Poisson mass function.

\begin{equation}\label{eq:coordinate-wise-prob}
\mathbb{P}\bigl(D_j(t)=0\bigr)
  =\sum_{k=0}^{\infty}
     \Bigl[\,
       \underbrace{e^{-\frac{2t}{d}}\frac{(\frac{2t}{d})^{2k}}{(2k)!}}_{\text{Poisson}\,(N_j(t)=2k)}
     \Bigr]
     \times
     \Bigl[\,
       \underbrace{\binom{2k}{k}2^{-2k}}_{\text{$k$ right,\ $k$ left}}
     \Bigr].
\end{equation}

By straightforward rearrangement, \ref{eq:coordinate-wise-prob} becomes a closed form expression involving \textit{a modified Bessel function of the first kind of order 0}, defined for us in \ref{eq: bessel}, and seen in Abramowitz-Stegun \cite[p.\,375]{abramowitz1964handbook}.

\begin{align}\label{eq:bessel-coordinate-wise}
    \mathbb{P}\bigl(D_j(t)=0\bigr)
    &=\sum_{k=0}^{\infty}e^{-\frac{2t}{d}}\frac{(\frac{2t}{d})^{2k}}{(2k)!}\binom{2k}{k}2^{-2k}\\
    &=\sum_{k=0}^{\infty} \frac{(\frac{t}{d})^{2k}}{(k!)^2} e^{-\frac{2t}{d}}\\&=I_0\biggl(\frac{2t}{d}\biggr) e^{-\frac{2t}{d}}
\end{align}

We must note that for the walks to collide, when $D(t)$ is 0, the walks must have exactly the same coordinates in every dimension, and since they are independent, we can exponentiate the individual coordinate probabilities by d. Thus, the collision probability is

\begin{equation}\label{eq:exponentiate}
    \mathbb{P}\biggl(D(t)=0\biggr) = \mathbb{P}\bigl(D_j(t)=0\bigr)^d .
\end{equation}

\section{The modified Bessel function}
In order to make an important substitution later in our analysis, we introduce a special function that has many variants. In this note, we focus on the modified Bessel function of the first kind of order 0.
\begin{defn}\label{eq: bessel}
    A \textbf{modified Bessel function of the first kind} of order 0 is defined as $$I_0(z) := \sum_{k=0}^\infty \frac{1}{(k!)^2} \left( \frac{z}{2} \right)^{2k}$$
\end{defn}

This identity appears in the Handbook of Mathematical Functions by Abramowitz and Stegun, p.~375 \cite{abramowitz1964handbook}, and is well defined for all z. We posit a proposition to ensure that our integrand is infinitely differentiable, which is crucial for our asymptotic analysis.

\begin{prop}\label{prop: smooth}
    \( I_0(z) \) is smooth and strictly positive for all real \( z \geq 0 \).
\end{prop}

\begin{proof}

  Absolute convergence of the series for all real $z$ implies smoothness; each term is positive, so $I_0(z)>0$.
\end{proof}

The Bessel function also admits an integral formula, also from \cite{abramowitz1964handbook}:

\begin{equation}\label{eq:cosine_integral}
    I_0(z) := \frac{1}{\pi}\int_{0}^\pi e^{z \cos(\theta)} \hspace{0.1cm} d\theta
\end{equation}
which we will obtain from our collision probability expression.

\section{Dimension threshold with Laplace approximation}

A quick lemma allows us to recharacterise our combinatorial term as the result of an integral, which motivates the Bessel function.

\begin{lem}\label{lem:cosine}
We have
    $$ \frac{1}{\pi}\int_0^{\pi}{\cos^{2k} {x} \hspace{0.1cm} d x} = \frac{1}{\pi}\int_{-\pi}^{\pi}{\cos^{2k} {x} \hspace{0.1cm} d x} = \binom{2k}{k} 2^{-2k}$$
\end{lem}

\begin{proof}
The reader can find the proof of this lemma in \cite{amdeberhan2016trig}. The exact proof is not important for our note, so we omit it for brevity.
\end{proof}

Now we prove theorem \ref{thm: theorem}.

\begin{proof}
    By Lemma \ref{lem:cosine}, we can rewrite \ref{eq:coordinate-wise-prob}, and noting that $\int_0^{\pi}{\cos^{j} {x} \hspace{0.1cm} d x}$ is 0 for odd j: 

\begin{align} \label{eq:intermediate}
    \mathbb{P}\biggl(D_j(t) = 0 \biggr) &= \sum_{k=0}^{\infty}\frac{(\frac{2t}{d})^{2k}}{{(2k)!}}e^{-{\frac{2t}{d}}} \cdot \frac{1}{\pi}\int_0^{\pi}{\cos^{2k} {x} \hspace{0.1cm} d x} \\
    &= \frac{e^{-\frac{2t}{d}}}{\pi}\int_0^\pi\sum_{k=0}^{\infty}\frac{(\frac{2t}{d}\cos(x))^{2k}}{(2k)!}\hspace{0.1cm} d x \\
    &= \frac{e^{-\frac{2t}{d}}}{\pi}\int_{-\pi}^\pi e^{\frac{2t}{d}\cos(x)}\hspace{0.1cm} d x \\
\end{align}   
Recognising that \ref{eq:cosine_integral} appears in the last equality in (\ref{eq:intermediate}) with $z = \frac{2t}{d}$, we observe that the integrand is maximised around $x=0$. Using Laplace's method, coupled with our smoothness argument from Proposition \ref{prop: smooth}, we can use replace $\cos(x)$ with its 2nd order Taylor expansion around $x=0$ to yield a Gaussian integral:
\begin{align}\label{eq:asympt}
    \frac{e^{-\frac{2t}{d}}}{\pi}\int_{-\infty}^\infty e^{\frac{2t}{d}(1 - \frac{1}{2} x^2 + O(x^4))}\hspace{0.1cm} d x 
    &=\frac{1}{\pi}\int_{-\infty}^\infty e^{-\frac{t}{d} x^2}\hspace{0.1cm} d x \\
    &= \sqrt{\frac{d}{\pi t}}
\end{align}

Utilising \ref{eq:exponentiate}, for the difference walk,

$$\mathbb{P}\biggl(D(t) = 0\biggr) = \mathbb{P}\biggl(D_j(t) = 0 \biggr)^d = \frac{C}{t^{\frac{d}{2}}} $$ for a constant C, and the expectation is the integral over all t, or

\begin{equation}\label{eq: result}
    \mathbb{E}\bigl[\#\bigl\{t: X(t)= Y(t) \bigr\}\bigr] = \int_1^{\infty}\frac{C}{t^{\frac{d}{2}}} dt
\end{equation}

A p-series test from analysis shows that this integral converges  when $\frac{d}{2} > 1$ or $d=3,4,\dots$ but diverges for $d=1,2$. Thus, the expected number of collisions is infinite for $d=1,2$, and finite for $d=3,4,\dots$ which fits with our current notions of transience and recurrence.

\end{proof}

The exact constant C is computable and we justify its value stated in \ref{thm: constant} here.

\begin{proof}
    From Equation \ref{eq:asympt} we have $\mathbb{P}\biggl(D_j(t) = 0 \biggr) \sim \sqrt{\frac{d}{\pi t}}$. Exponentiating due to the number of dimensions from \ref{eq:exponentiate} we recover our result of $(\frac{d}{\pi})^{\frac{d}{2}}$.
\end{proof}

\begin{rem}
    It may be interesting to analyze collisions of random walks that are not independent of each other, or use these techniques to study coalescing or self-avoiding random walks. 
\end{rem}

\begin{rem}
    There do exist recurrent graphs where there are finite collisions; see \cite{krishnapur2004recurrentgraphsindependentrandom}, the infinite collision property for three-dimensional uniform spanning trees \cite{doi:10.1142/S2661335223500053}. Generally walks on non-lattice graphs do not follow this low dimension/infinite relation; see \cite{Hutchcroft_2015}
\end{rem}

\bibliography{mybib}
\bibliographystyle{alpha}

\section*{Acknowledgements}
I would like to thank Susan Ruff, Johannes Hosle, and Korina Digalaki for their valuable feedback and discussions that significantly improved this paper.
\end{document}